\documentclass[12pt,letterpaper]{amsart}

\usepackage{multirow}
\usepackage{graphicx}

\newcommand{\del}[1]{\frac{\partial}{\partial #1}}

\newtheorem{theorem}{Theorem}
\newtheorem{lemma}[theorem]{Lemma}
\newtheorem{prop}[theorem]{Proposition}
\newtheorem*{theorema}{Theorem A}
\newtheorem*{theoremb}{Theorem B}
\newtheorem*{theoremc}{Theorem C}

\theoremstyle{remark}
\newtheorem{rema}[theorem]{Remark}
\newtheorem{exem}[theorem]{Example}
\newtheorem{defi}[theorem]{Definition}

\title[Vector fields, separatrices and Kato surfaces]{Vector fields, separatrices and Kato surfaces}
\author{Adolfo Guillot}\address{Instituto de Matem\'aticas, Unidad Cuernavaca,
 Universidad Nacional Aut\'onoma de M\'exico,
A.P.~273-3 Admon.~3,
Cuernavaca, Morelos, 62251, M\'exico}
\email{adolfo.guillot@im.unam.mx}

\thanks{Partially supported by Conacyt-Mexico grant no. 167594.}
\thanks{To appear in the \emph{Annales de l'Institut Fourier (Grenoble), 2013}}
\dedicatory{to the memory of Marco Brunella}

\begin{document}
\begin{abstract}We prove that a singular complex surface that admits a complete holomorphic vector field that has no invariant curve  through a singular point of the surface is obtained from a Kato surface by contracting some divisor (in particular, it is compact). We also prove that, in a singular Stein surface endowed with a complete holomorphic vector field, a singular point of the surface where the zeros of the vector field do not accumulate is either a quasihomogeneous or a cyclic quotient singularity. We give new proofs of some results  concerning the classification of compact complex surfaces admitting holomorphic vector fields. Our proofs rely in a combinatorial description of the vector field on a resolution of the singular point based on previous work of Rebelo and the author.  \\

\noindent \textsc{R\'esum\'e}. On prouve qu'un espace analytique complexe de dimension deux admettant un champ de vecteurs complet qui n'a pas de s\'eparatrice passant par un point singulier de la surface peut s'obtenir \`a partir d'une surface de Kato en effondrant un diviseur (en particulier, l'espace est compact). On prouve que, dans un espace analytique de Stein de dimension deux muni d'un champ de vecteurs complet, un point singulier de l'espace qui est un point d'\'equilibre isol\'e du champ est soit une singularit\'e quasi-homog\`ene, soit une singularit\'e de Klein. On red\'emontre quelques r\'esultats concernant la classification des surfaces complexes compactes admettant des champs de vecteurs holomorphes. Les preuves reposent sur des travaux r\'ecents de Rebelo et de l'auteur donnant une description combinatoire des champs de vecteurs complets.\end{abstract}

\maketitle

\clearpage 

\section{Introduction}
In the realm of ordinary differential equations in the complex domain, Briot  and Bouquet studied the differential equations of the  form
$$\frac{\partial y}{\partial x}=\frac{f(x,y)}{g(x,y)},$$
where~$f$ and~$g$ are holomorphic functions vanishing at the origin and without common factors. They gave conditions guaranteeing the  existence of holomorphic solutions~$y(x)$ such that~$y(0)=0$~\cite{briotbouquet}. The problem of Briot and Bouquet can be stated, more generally, as the problem of existence of curves~$\phi:(\mathbf{C},0)\to(\mathbf{C}^2,0)$ that are tangent to the foliation generated by the one-form
\begin{equation}\label{bb}f(x,y)dx-g(x,y)dy,\end{equation}
with the original problem corresponding to curves of the form~$\phi(x)=(x,y(x))$. A definitive solution to the problem of Briot and Bouquet was given by Camacho and Sad~\cite{CS-separatrix}, who proved the existence of invariant curves for every foliation of the form~(\ref{bb}). The theorem of Camacho and Sad motivated the quest for extensions of their result to other ambient spaces and foliations by leaves of other dimensions.

For foliations defined in normal complex analytic spaces of dimension two, Camacho proved that a separatrix through a singular point~$p$ exists if the dual graph of the exceptional divisor of a resolution of~$p$ is a tree~\cite{camacho-singular}. In the same article, Camacho exhibited foliations  without separatrices in germs of singular analytic surfaces. In contrast to the situation in~$(\mathbf{C}^2,0)$, not every foliation on a singular analytic surface is locally induced by a vector field, so we may ask the following question, attributed to G\'omez-Mont: \emph{Does a vector field in a singular surface have a separatrix passing through the singular point?}

Natural examples answering negatively this question  arise from  holomorphic vector fields on \emph{intermediate Kato surfaces}. Kato surfaces are non-singular compact complex surfaces; they are  minimal, non-K\"ahler and
belong to the class~$\mathrm{VII}$ in the Enriques-Kodaira Classification. They were introduced by Kato~\cite{kato} and, among them, we find \emph{intermediate} ones~\cite{dloussky-kato}. Intermediate Kato surfaces have a divisor~$D$ (the maximal reduced divisor of rational curves) having connected support and a negative-definite intersection form. Some of these admit a  holomorphic vector field~$X$. By the negative-definiteness of its intersection form, $D$ is preserved by~$X$. These vector fields are well-understood and we know that, close to~$D$, there are no singularities of~$X$ off~$D$ and that any germ of
curve invariant by the vector field is contained in~$D$~\cite[Lemme~2.2]{dloussky-vf1}. Hence, when contracting~$D$ to a point~$p$, we obtain a two-dimensional complex analytic space  endowed with a holomorphic vector field having an isolated equilibrium point at~$p$ which \emph{does not} have a separatrix. (By the compactness of the space, this vector field is complete).\\

These examples of vector fields   without separatrices in complex analytic spaces of dimension two are unique within the class of vector fields in \emph{compact} ones (it is probably easy to deduce this from the classification of holomorphic vector fields on compact complex surfaces~\cite[Thm.~0.3]{dloussky-vf2}). Our main result affirms that they are also unique within the larger class of \emph{complete} vector fields:

\begin{theorema} Let~$S$ be a connected, normal, irreducible, complex, two-dimensional analytic space with a singularity at~$p\in S$ and let~$X$ be a complete holomorphic vector field on~$S$. If the foliation induced by~$X$ has no separatrix through~$p$, the minimal resolution of~$S$ at~$p$ is a Kato surface.
\end{theorema}
We stress the fact that the only global assumption on the surface is the completeness of the vector field and that the compactness of~$S$ is a consequence of our result. It also implies that~$p$ is the only singular point of~$S$ and that~$p$ is an isolated equilibrium point of~$X$.

Theorem~A guarantees the existence of separatrices for complete vector fields on two-dimensional analytic spaces which are, for example, not compact. For Stein spaces, this result can be strengthened:
\begin{theoremb} Let~$S$ be a normal, irreducible  complex two-dimensional Stein space and~$X$ a complete holomorphic vector field on~$S$. Let~$p$ be a singular point of~$S$ that is an isolated equilibrium point of~$X$. Either
\begin{itemize}
 \item there are either one or two separatrices of~$X$ through~$p$ and~$(S,p)$ is a cyclic quotient singularity; or
 \item there is an infinite number of separatrices of~$X$ through~$p$ and~$X$ induces an action of~$\mathbf{C}^*$.
\end{itemize}
\end{theoremb}

The classification of holomorphic vector fields on compact complex surfaces was achieved by Dloussky, Oeljeklaus and Toma in~\cite[Thm.~0.3]{dloussky-vf2}. The last piece in this classification is given by vector fields in Kato surfaces. The natural way in which Kato surfaces appear in the proof of Theorem~A will allow us to give an alternative proof of some parts of the classification of holomorphic vector fields on compact complex surfaces, as suggested by Matei Toma. In a relatively self-contained manner and without relying on the Enriques-Kodaira Classification, we will, by borrowing some arguments from~\cite{dloussky-vf0}, prove the following result:

\begin{theoremc}Let~$X$ be a non-trivial holomorphic vector field with zeros on the minimal compact complex surface~$S$ inducing an effective action of~$\mathbf{C}$. At least one of the following holds:
\begin{enumerate}
\item   $S$ is rational or ruled.
\item  $X$ has a first integral.
\item  There is an effective divisor~$Z$ in~$S$ such that~$Z^2=0$.
\item  $S$ is a Kato surface.
\end{enumerate}
\end{theoremc}

Our results are based in a local analysis of the vector field in a neighborhood of some invariant divisors of the surface (the exceptional divisor of a resolution in Theorems~A and~B, the maximal invariant divisor in Theorem~C), where we use the notion of \emph{semicompleteness} to exploit, locally, the  completeness assumption on the vector  field. We make extensive use of the ``bimeromorphic'' theory of semicomplete vector fields on surfaces developed jointly with Rebelo in~\cite{guillot-rebelo}.\\

This article is organized as follows. In Section~\ref{sec-localmodels} we will recall some standard facts about vector fields and foliations on complex surfaces and analytic spaces (the reader is however supposed to be familiar, for example, with the material discussed in the first Chapter of~\cite{brunella} and in~\cite[Ch.~I, \S8]{BPV}). We also discuss semicomplete vector fields on surfaces based mainly on~\cite{guillot-rebelo}. Section~\ref{sec-cycles} describes the combinatorics of semicomplete vector fields in the neighborhood of divisors containing cycles. Theorem~A will be proved in Section~\ref{sec:kato}. It relates the resolution of a singularity admitting a semicomplete vector field without separatrices to the construction of Kato surfaces. Theorems~B and~C will be respectively proved in Sections~\ref{sec:stein} and~\ref{sec:comp}.\\

The author heartily thanks Patrick Popescu-Pampu, Julio Rebelo, Jos\'e Seade, Jawad Snoussi, Matei Toma, Meral Tosun and the anonymous referee.

\section{Generalities}\label{sec-localmodels}

\subsection{Vector fields and foliations}

Let~$S$ be complex (non-singular) surface and $X$ a holomorphic vector field
in~$S$. Around every point~$q$ in~$S$, $X$ may be locally written as~$fY$
where~$f$ is a holomorphic function and~$Y$ is a holomorphic vector field with
isolated singularities, well-defined up to multiplication by a non-vanishing
holomorphic function. The locally defined vector field~$Y$ defines a globally
well-defined foliation~$\mathcal{F}$, \emph{the foliation induced by~$X$}, that
will be denoted by~$\mathcal{F}_X$ whenever we need to stress its relation
to~$X$. If~$Y(p)\neq 0$, we say that~$p$ is a \emph{regular} point
of~$\mathcal{F}$ and we say that~$p$ is a \emph{singularity}
of~$\mathcal{F}$ otherwise.  The foliation~$\mathcal{F}$ is said to be
\emph{reduced in Seidenberg's sense} if the linear part of~$Y$ at a
singularity~$p$ of~$\mathcal{F}$ is non-nilpotent and, if it has two
non-vanishing eigenvalues, if their ratio is not a positive rational.
Seidenberg's Theorem affirms that every foliation may be brought, in a locally
finite number of blowups, to one where every singularity is reduced (see, for 
example, \cite{cs-livro}). We say that the curve~$\gamma\subset S$ with reduced
equation~$g$ is \emph{invariant by~$X$} if~$g$ divides~$X\cdot g$. Moreover,
if~$X=fY$ and~$g$ divides~$Y\cdot g$ we say that~$\gamma$ is \emph{invariant
by~$\mathcal{F}$}  (if a curve is invariant by~$\mathcal{F}$ it is also
invariant by~$X$ but the converse need not be true). In this last case, if~$g$
is irreducible and~$f=g^r h$ with~$g$ and~$h$ relatively prime, we say
that~$r\in\mathbf{Z}$ is the \emph{order of~$X$ along~$\gamma$}, and
write~$\mathrm{ord}(X,\gamma)=r$. 

If~$C$ is a compact curve invariant by~$\mathcal{F}$, its self-intersection can be
calculated by means of the Camacho-Sad formula: for
each singularity~$p_1,\ldots, p_m$, of~$\mathcal{F}$ lying in~$C$ the
\emph{Camacho-Sad} index, $\mathrm{CS}(\mathcal{F},C,p_i)\in\mathbf{C}$, is
defined and the Camacho-Sad formula yields
$C\cdot C=\sum_i \mathrm{CS}(\mathcal{F},C,p_i)$~\cite{CS-separatrix}.\\

The vector field~$X$ endows the foliation~$\mathcal{F}$ with a \emph{leafwise
affine structure} (with singularities) varying holomorphically in the transverse
direction: every curve~$\gamma$ invariant by~$\mathcal{F}$ inherits from~$X$ an
\emph{affine structure} (with singularities)~\cite[Prop.~8]{guillot-rebelo}. In
the case where the vector field does not vanish along the curve, the charts of
this structure are given by the inverses of the local solutions of the vector
field~\cite[\S3.1]{guillot-rebelo}. The affine structure affects every
point~$p\in \gamma$ with a \emph{ramification index}
$\mathrm{ind}(C,p)\in\mathbf{C}^*\cup\{\infty\}$ \cite[Def.~4]{guillot-rebelo},
whose value is~$1$ except for a discrete (with the plaque topology) set of
points in~$\gamma$, the \emph{singularities} of the affine structure. If~$C$ is
a compact curve invariant by~$\mathcal{F}$ of Euler characteristic~$\chi(C)$, we
have, for the above indexes, the Poincar\'e-Hopf
relation~\cite[Prop.~5]{guillot-rebelo}:
\begin{equation}\label{poincarehopf}\chi(C)=\sum_{p\in C} 1-\frac{1}{\mathrm{ind}(C,p)}.\end{equation}

For vector fields, we have the following definition~\cite[Def.~16]{guillot-rebelo}:

\begin{defi}\label{def:reduced} A holomorphic vector field~$X$ on a surface~$S$ is said to be \emph{reduced} if the induced foliation~$\mathcal{F}_X$ is reduced in Seidenberg's sense and if for every point~$p$, the union of all the curves containing~$p$ that are invariant by~$X$ is a curve with normal crossings. A couple~$(X,D)$ of a holomorphic vector field~$X$ and a  divisor~$D$ invariant by~$X$ is said to be \emph{minimal good} if~$X$ is reduced in a neighborhood of~$D$, if every irreducible component of~$D$ is non-singular and if no exceptional curve of the first kind belonging to~$D$ may be collapsed while keeping~$X$ reduced and the corresponding divisor non-singular. \end{defi}

Every holomorphic vector field may be transformed, by a locally finite number of blowups, to a \emph{reduced} one (combine Seidenberg's Theorem with the resolution of embedded curves in surfaces).

\subsection{Vector fields in analytic  spaces}

Let~$(S,\varpi)$ be a germ of irreducible, reduced, complex-analytic two-dimen\-sional space. By a holomorphic vector field on~$(S, \varpi)$ we mean, indistinctly, either a derivation of the local ring~$\mathcal{O}_{S,\varpi}$ or, for an embedding~$j:(S,\varpi)\to (\mathbf{C}^n,0)$, the restriction to~$j(S)$ of a holomorphic vector field in~$\mathbf{C}^n$ tangent to~$j(S)$ \cite[\S3]{rossi}. By the following Proposition (obtained with the help of Jawad Snoussi), a holomorphic vector field on~$(S,\varpi)$  is also equivalent to a holomorphic vector field on a resolution.

\begin{prop}\label{desing} Let~$X$ be a holomorphic vector field in the germ of two-dimensional irreducible analytic space~$(S,\varpi)$. Let~$M:(S_\mu,D_\mu)\to(S,\varpi)$ be the minimal resolution. There exists a holomorphic vector field~$X_\mu$ in~$S_\mu$ such that~$M_*X_\mu=X$.
\end{prop}

\begin{proof} Zariski proved that a resolution of~$(S,\varpi)$ may be obtained by alternating two procedures: normalization and the blowing up of singular points in normal surfaces (see~\cite{zariski} for the algebraic case, \cite{bondil-le} for the analytic one). In order to prove the existence of a resolution where the preimage of the vector field extends holomorphically to the exceptional divisor, it suffices to show that these two procedures transform holomorphic vector fields into holomorphic ones. 

We begin with normalization. Suppose, up to separating the irreducible components of~$S$, that~$S$ is irreducible at~$\varpi$. Let~$\mathcal{O}$ be the local ring of holomorphic functions at~$\varpi$, $F$ its field of fractions and~$\overline{\mathcal{O}}\subset F$ the integral closure of~$\mathcal{O}$. Let~$d:\mathcal{O}\to \mathcal{O}$ be the derivation induced by~$X$. This derivation extends to~$d_F:F\to F$. Since~$\mathcal{O}$ is a Noetherian integral domain containing~$\mathbf{Q}$, a theorem of Seidenberg~\cite{seidenberg-integral} guarantees that~$d_F(\overline{\mathcal{O}})\subset \overline{\mathcal{O}}$. Hence, if~$\pi:(\overline{S},\overline{\varpi})\to(S,\varpi)$ is the normalization, there exists a holomorphic vector field in~$(\overline{S},\overline{\varpi})$ mapping to~$X$ via~$\pi$. This proves that  normalization transforms holomorphic vector fields into holomorphic ones. 

Let us now deal with blowups. Suppose that~$S$ is normal at the singular
point~$\varpi$ and let~$j:(S,\varpi)\to(\mathbf{C}^n,0)$ be an embedding.
Let~$Y$ be a holomorphic vector field in~$(\mathbf{C}^n,0)$ that restricts
to~$j_*X$ in~$j(S)$. Since~$S$ is normal at~$\varpi$, $0$ is an isolated
singularity of~$j(S)$ and must be preserved by~$Y$. Hence, $Y$ vanishes at
the origin of~$\mathbf{C}^n$. Upon blowing up the latter, $Y$ extends as a
holomorphic vector field to the exceptional divisor. This proves that the blowup
of singular points in normal analytic surfaces transforms holomorphic vector
fields into holomorphic ones. 

In consequence, in Zariski's resolution of~$(S,\varpi)$, $X$ is transformed as a holomorphic vector field. Since the minimal resolution of~$(S,\varpi)$ may be obtained from Zariski's one by contracting exceptional curves of the first kind and since this procedure maps holomorphic vector fields to holomorphic ones, the Proposition is proved.
\end{proof}

\begin{defi}\label{reducedvf} Let~$X_0$ be a germ of holomorphic vector field in the germ of normal two-dimensional analytic space~$(S_0,\varpi)$.
A \emph{resolution} $\pi:(S,D,X)\to(S_0,\varpi,X_0)$ is a resolution
$\pi:(S,D)\to(S_0,\varpi)$ and a holomorphic vector field~$X$ on~$S$, reduced in
the sense of Definition~\ref{def:reduced}, such that~$\pi_*(X)=X_0$. A resolution $\pi:(S,D,X)\to(S_0,\varpi,X_0)$ is said to be \emph{minimal good} if~$(X,D)$ is minimal good in the sense of Definition~\ref{def:reduced}.
\end{defi}

\begin{prop}  Let~$X_0$ be a germ of holomorphic vector field in the germ of normal two-dimensional analytic space~$(S_0,\varpi)$. Then it admits a minimal good resolution~$\pi:(S,D,X)\to(S_0,\varpi,X_0)$ in the sense of Definition~\ref{reducedvf}.
\end{prop}
\begin{proof}
Let~$\pi:(S,D)\to (S_0,\varpi)$ be a resolution of the analytic space. By
Proposition~\ref{desing}, there exists a holomorphic vector field~$X$ on~$S$  
such that~$\pi_*X=X_0$. We may, by performing finitely many blowups, make the
resulting vector field on~$S$ reduced in the sense of
Definition~\ref{def:reduced} and, afterwards, desingularize, if necessary, the
irreducible components of~$D$. If the resulting resolution is not minimal good,
it may be rendered so by suitably collapsing some of the (finitely many)
irreducible components of~$D$. \end{proof}

\subsection{Semicompleteness in manifolds}\label{semicomplete}

In open manifolds, characterizing complete holomorphic vector fields is not an easy task. In~\cite{rebelo}, Rebelo introduced the class of \emph{semicomplete} holomorphic vector fields, a class containing complete vector fields that is, in many senses, better behaved.

A holomorphic vector field~$X$ in a complex manifold~$M$ induces an ordinary differential equation in the complex domain. The existence and uniqueness theorem for such equations guarantees that for every initial condition~$p\in M$ there exists some domain~$U\subset \mathbf{C}$, $0\in U$, and a map~$\phi:(U,0)\to (M,p)$ solving the differential equation. The vector field is \emph{complete} if for every $p\in M$ we can find a solution~$\phi:(\mathbf{C},0)\to (M,p)$. There are essentially two (not independent) conditions that a vector field must fulfill in order to be complete:
\begin{itemize} \item The analytic continuation of the solutions of the induced
differential equation should not present multivaluedness (this allows for each 
solution to be defined in a maximal subset of~$\mathbf{C}$).
 \item This maximal subset of~$\mathbf{C}$ must be~$\mathbf{C}$.
\end{itemize}
Semicomplete vector fields are those satisfying the first condition. One of their main properties is that semicomplete vector fields remain semicomplete when restricted to any open subset.
In particular, it makes perfect sense to speak of \emph{germs} of semicomplete vector fields, or to study semicomplete vector fields in the neighborhood of a curve, establishing local and semi-local obstructions for a vector field to be complete. \\

Germs of semicomplete holomorphic vector fields may be described up to
biholomorphism by a list of local models, as done in~\cite{ghys-rebelo},
\cite{rebelo-mex} and \cite{rebelo-realisation}. Semicompleteness is preserved
by the bimeromorphic transformations preserving the holomorphicity of the vector
field~\cite[Cor.~12]{guillot-rebelo}, and it thus makes sense to speak about
local models for germs of \emph{reduced semicomplete} holomorphic vector fields.
These local models were reobtained and refined in~\cite[\S5]{guillot-rebelo}
and are presented in Table~\ref{table:combinatorics}. \\

In this Table we have the local model of every reduced semicomplete vector field~$X$ and, for every curve invariant by~$\mathcal{F}_X$, the order of~$X$ along the curve, the ramification index of the affine structure and the Camacho-Sad index of~$\mathcal{F}_X$. We do not claim that every vector field having such a local model is semicomplete, but rather that every germ of reduced semicomplete vector field has, in convenient coordinates, one of the local models appearing in the Table. Let us comment briefly its contents.

The first three lines correspond to points where the foliation is non-singular. The vector field may have a zero of arbitrary order along one leaf and zeros up to order two along a curve transverse to the foliation. 
At an \emph{affine regular} point (called simply \emph{regular} in~\cite{guillot-rebelo}), the induced affine structure is non-singular (it has ramification index equal to~$1$).

The next three lines correspond to the local models where the induced foliation has an isolated singularity, induced by a vector field with no zero eigenvalues. The \emph{singular non-degenerate} case corresponds to the case where the vector fields has an isolated singularity and the other two, to the cases where there are zeros of the vector field along (at least) one of the separatrices of the foliation. In these two cases, the ramification index of the affine structures of the separatrices is simultaneously finite or infinite. For the \emph{finite ramification} points, the local model is a true normal form (in the sense that it has no inessential parameters) and all these germs are semicomplete. Furthermore, at such points, the orders of~$X$ over the two separatrices~$C_1$ and~$C_2$  and the ramification indices of the affine structures at their intersection point~$p$ satisfy the \emph{reciprocity relation}
\begin{equation}\label{reciprocity}\mathrm{ord}(X,C_1)\mathrm{ind}(C_2,p)+\mathrm{ord}
(X,C_2)\mathrm{ind}(C_1,p)=-1.\end{equation}
The affine structure may be non-singular for a separatrix through a \emph{finite ramification} point (the case where one of the ramification indices is equal to~$1$ is not excluded).

The last line corresponds to the case where the foliation is generated by a holomorphic vector field with isolated singularity having one vanishing and one non-vanishing eigenvalue (the \emph{saddle-node} case). The singularity of the vector field is necessarily isolated~\cite[Lemme~3.2]{rebelo-mex} and semicompleteness imposes serious constrains at the level of the foliation, as established in~\cite[Thm.~4.1]{rebelo-realisation}, providing the local model appearing on the Table.\\

\begin{table}
\begin{tabular}{|c|c|c|c|c|}\hline
type   & local model & ord & ind & CS \\ \hline \hline
\begin{tabular}{cc}affine \\ regular\end{tabular}   & $\displaystyle y^q\del{x}$ & $q$ & $1$ & $0$ \\ \hline
  single zero &  $\displaystyle f(x,y)x y^q\del{x}$ & $q$ &  $\infty$ & 0 \\  \hline
 double zero &  $\displaystyle f(x,y)x^2 y^q\del{x}$ & $q$ & $-1$ & 0 \\ \hline
\multirow{4}{*} {\begin{tabular}{c} singular \\ non-degenerate \end{tabular}} & \multirow{4}{*} {\begin{tabular}{c} $\displaystyle
x(\lambda+\cdots)\del{x}+y(\mu+\cdots)\del{y}$ \\  $\lambda/\mu\notin\mathbf{Q}^+$ \end{tabular} } & \multicolumn{3}{|c|}{$x=0$}   \\ \cline{3-5}
  &  & $0$ & $\infty$ & $ \lambda/\mu$   \\ \cline{3-5}
 &  &  \multicolumn{3}{|c|}{$y=0$} \\ \cline{3-5}
 &  & $0$ & $\infty$ & $ \mu/\lambda$   \\ \cline{1-5}
 \multirow{4}{*} {\begin{tabular}{c}finite  \\ ramification \end{tabular}} & \multirow{4}{*} {\begin{tabular}{c} $\displaystyle
x^py^q\left(mx\del{x}-ny\del{y}\right)$ \\ $pm-qn=1$ \end{tabular}} & \multicolumn{3}{|c|}{$x=0$}   \\ \cline{3-5}
  &  & $p$ & $n$ & $-m/n$   \\ \cline{3-5}
 &  &  \multicolumn{3}{|c|}{$y=0$} \\ \cline{3-5}
 &  & $q$ & $-m$ & $-n/m$   \\ \cline{1-5}
 \multirow{4}{*} {\begin{tabular}{c}infinite  \\ ramification  \end{tabular}} & \multirow{4}{*} { $\displaystyle
x^py^q\left(x[q+\cdots]\del{x}-y[p+\cdots]\del{y}\right)$} & \multicolumn{3}{|c|}{$x=0$}   \\ \cline{3-5}
  &  & $p$ & $\infty$ & $-q/p$   \\ \cline{3-5}
  &  &  \multicolumn{3}{|c|}{$y=0$} \\ \cline{3-5}
 &  & $q$ & $\infty$ & $-p/q$  \\  \hline
 \multirow{4}{*}   {saddle-node}    & \multirow{4}{*} {\begin{tabular}{c} $\displaystyle f(x,y)\left(x[1+\nu y]\del{x} +y^2\del{y}\right)$ \\ $\nu\in\mathbf{Z}$ \end{tabular}} & \multicolumn{3}{|c|}{$x=0$}   \\ \cline{3-5}
 &  & $0$ & $-1$ & $\nu$   \\ \cline{3-5}
 &  &  \multicolumn{3}{|c|}{$y=0$} \\ \cline{3-5}
 &  & $0$ & $\infty$ & $0$   \\ \cline{1-5}
\end{tabular}
\caption{Local models of reduced holomorphic vector fields. In these, $p,q\geq 0$, $m,n>0$,$\lambda,\mu\in\mathbf{C}^*$ and~$f$ is a non-vanishing holomorphic function. In the first three, the invariant curve is given by~$\{y=0\}$.}\label{table:combinatorics}
\end{table}

An important fact behind the classification of local models is that, for a curve $\gamma$ invariant by~$\mathcal{F}_X$, the previously mentioned  affine structure is \emph{uniformizable} (in the complement of the singular points and as a curve with an affine structure, $\gamma$ is the quotient of a subset of~$\mathbf{C}$ by a group of affine transformations)
if~$X$ is semicomplete in a neighborhood of~$\gamma$~\cite[\S3.2]{guillot-rebelo}. This implies that, for every~$p\in\gamma$, $\mathrm{ind}(\gamma,p)\in\mathbf{Z}^*\cup\{\infty\}$~\cite[Prop.~6]{guillot-rebelo}.

In particular, if~$X$ is semicomplete and if~$C$ is a compact curve invariant by~$\mathcal{F}_X$, from the Poincar\'e-Hopf relation~(\ref{poincarehopf}), $\chi(C)\geq 0$. If~$C$ is an elliptic curve, the ramification index is everywhere equal to~$1$ (the affine structure has no singularities). If~$C$ is a rational curve with an affine structure having singularities at the points~$p_1,\ldots,p_r\in C$ and~$\mathrm{ind}(C,p_j)=i_j$, by the above formula, 
\begin{equation}\label{sumi}\sum_{j=1}^r \frac{1}{i_j}=r-2.\end{equation}
In this case, we will say that the affine structure is of type~$(i_1,\ldots, i_r)$. The only possible types of uniformizable affine structures are $(-1)$, $(n,-n)$ for $n\geq 2$, $(\infty,\infty)$, $(2,2,\infty)$,  $(2,3,6)$,  $(2,4,4)$,  $(3,3,3)$ and  $(2,2,2,2)$. Let us sketch a proof of this. Since~$1/i_j\leq 1/2$, the left hand side of~(\ref{sumi}) is smaller or equal than~$r/2$ , which implies that~$r/2\geq r-2$ and thus~$r\leq 4$. For~$r=4$, equality holds and thus~$i_j=2$ for every~$i$. For~$r=3$, if~$1/i_1\leq 0$, $1/i_2+1/i_3\geq 1$ (but~$i_2^{-1}$ and~$i_3^{-1}$ are at most equal to~$1/2$) and  we must have~$i_1=\infty$, $i_2=2$, $i_3=2$. The remaining cases are straightforward and dealt with in the same way.

The only uniformizable affine structures in rational curves globally induced by holomorphic vector fields are those of type~$(\infty,\infty)$, if the vector field has two singular points, and~$(-1)$, if it has only one. Rational curves endowed with an affine structure of type~$(n,-n)$ will be simply called \emph{rational orbifolds of order~$n$} and by \emph{rational orbifolds of order~$1$} we will refer to those of type~$(-1)$. 
\begin{rema}\label{ratorb}
Rational orbifolds are the only uniformizable affine structures on curves having at least one singularity with a negative ramification index.
\end{rema}

\subsection{Semicompleteness in analytic spaces}\label{ssscias} The notion of
semicomplete vector field, defined originally for manifolds, extends directly to
analytic spaces. Semicompleteness is, again, preserved under restrictions to
open subsets.

Let $S_0$ be a two-dimensional analytic space and~$X_0$ a vector field in~$S_0$. Let~$\varpi\in S_0$ be a singular point. By the previous remark, $X_0$ is semicomplete in~$S_0\setminus\{\varpi\}$ if it is semicomplete in~$S_0$. Let~$\pi:(S,D,X)\to(S_0,\varpi,X_0)$ be a resolution. Since~$\pi|_{S\setminus D}:S\setminus D\to S_0\setminus\{\varpi\}$ is a biholomorphism mapping~$X$ to~$X_0$, $X$ is semicomplete in~$S\setminus D$ if and only if~$X_0$ is semicomplete in~$S_0\setminus\{\varpi\}$.  For a vector field defined in a manifold, its \emph{multivaluedness locus}  (the subset where it fails to be semicomplete) is open~\cite[Cor.~12]{guillot-rebelo}. Hence, $X$ will be semicomplete in~$S\setminus D$ if and only if it is semicomplete in~$S$. 

Hence, if~$X_0$ is semicomplete in~$S_0$ then $X$ is semicomplete in~$S$, and
may be studied with the tools previously described.

\section{Cycles of invariant curves in semicomplete vector fields}\label{sec-cycles}
We will begin by studying the nature of the divisors that are invariant by a semicomplete vector field and that do not have other invariant curves. Recall that, to a divisor~$D$ in a surface, we may associate a \emph{dual graph}, consisting of a vertex for each irreducible component and an edge for each point of intersection of two irreducible components. 

\begin{prop}\label{propcycles} Let~$S$ be a non-singular surface, $X$ a reduced semicomplete holomorphic vector field on~$S$ and~$D\subset S$ a connected divisor invariant by~$\mathcal{F}_X$ such that~$(X,D)$ is minimal good. Let~$\Gamma$ be the dual graph of~$D$. If~$\Gamma$ has a cycle and every curve invariant by~$\mathcal{F}_X$ intersecting~$D$ is contained in~$D$, then:
\begin{itemize}
\item Every irreducible component of~$D$ is rational.
\item The cycle in~$\Gamma$ is unique.
\item If~$\Gamma$ reduces to the cycle, $D$ supports an effective divisor of vanishing self-intersection.
\item If~$\Gamma$ does not reduce to the cycle, the only vertices in~$\Gamma$ with degree greater than two have degree three and belong to the cycle. Every irreducible component of~$D$ is, as a curve with an affine structure, a rational orbifold.
\end{itemize}
\end{prop}

In the aim of resorting solely to the items of the previous section, we could not help overlapping with some of the arguments and results in~\cite{guillot-rebelo}.\\

At every point of~$D$, the vector field is locally of one of the forms appearing in Table~\ref{table:combinatorics}. The only singularities of~$D$ in the reduced semicomplete vector field~$X$ are normal crossings where the two local branches are contained in different irreducible components.  Each of these irreducible components is  either a rational or an elliptic curve, following the discussion in~\S\ref{semicomplete}. By hypothesis, $\Gamma$ is connected and every edge joins two different vertices. Let~$\Gamma_0$ be a cycle of~$\Gamma$ of the form
\begin{equation}\label{zykel}C_0\stackrel{p_0}{\text{---}}C_1\stackrel{p_1}{\text{---}}\cdots
\stackrel{p_{l-1}}{\text{---}} C_l=C_0, \end{equation}
meaning that the irreducible components of~$\Gamma_0$ are~$C_0 \ldots, C_{l-1}$
and that~$C_i$ intersects~$C_{i+1}$ transversely at the point~$p_i$. For
each~$p_i$, the local model is either a \emph{finite ramification}, an
\emph{infinite ramification}, a \emph{singular non-degenerate} point or a
\emph{saddle-node}. One of the following holds:
\begin{itemize}
\item There is a point~$p_i$ where the local model of~$X$ is either a
\emph{finite ramification} point or a \emph{saddle-node}.  In both cases, the
intersection point has a negative ramification index for one of the invariant
curves. Suppose that $\mathrm{ind}(C_1,p_0)<0$. Hence (Remark~\ref{ratorb}),
$C_1$ is a rational orbifold and~$\mathrm{ind}(C_1,p_1)>0$. From
Table~\ref{table:combinatorics}, $p_1$ is necessarily a \emph{finite
ramification} point and~$\mathrm{ind}(C_2,p_1)<0$. Continuing this argument we
conclude that every~$C_i$ is a rational orbifold and that \emph{the local model of~$X$
at every~$p_i$ is a finite ramification point}.
\item The local model of~$X$ at~$p_0$  is an \emph{infinite ramification} point
(Table~\ref{table:combinatorics}). This means that~$X$ vanishes along~$C_1$ and
thus~$p_1$ is necessarily an \emph{infinite ramification} point. Continuing
this argument, we conclude that \emph{$p_i$ is an infinite ramification point
for every~$i$}.  
\item The local model of~$X$ at~$p_i$  is a  \emph{singular non-degenerate}
point for every~$i$. 
\end{itemize}
In the last two cases, for every~$i$,  $C_i$
has two points of ramification index~$\infty$ and is thus of
type~$(\infty,\infty)$.

Hence, in all cases, $C_i$ is a rational curve. Moreover, for every~$i$ and every~$q\in C_i$, $q\notin\{p_i,p_{i-1}\}$,
\begin{equation}\label{posram}\mathrm{ind}(C_i,q)>0.\end{equation}
Consider, within~$\Gamma$, the path
$C_1\stackrel{q_1}{\text{---}}E_2\stackrel{q_2}{\text{---}}\cdots
\stackrel{q_{l-1}}{\text{---}}
E_l$ such that~$C_1\in\Gamma_0$
but~$E_2\notin\Gamma_0$. Since~$\mathrm{ind}(C_1,q_1)>0$, $q_i$ is a finite ramification point and~$\mathrm{ind}(E_2,q_1)<0$. Hence,  (Remark~\ref{ratorb}) $E_2$ is a rational orbifold. This implies that~$\mathrm{ind}(E_2,q_2)>0$. Continuing
this argument, we conclude that $\mathrm{ind}(E_n,q_{n-1})<0$. In particular,
$E_n$ cannot belong to a cycle: since~$\Gamma$ is connected, the cycle
in~$\Gamma$ is unique. 

\paragraph*{If~$\Gamma$ reduces to the cycle } Suppose that $\Gamma$ is of the form~(\ref{zykel}). Let~$k_i=-C_i^2$. We have, by the minimal good hypothesis on~$(X,D)$, two cases:
\begin{itemize}
 \item  $k_i\geq 2$ for every~$i$. If~$k_i>2$ for some~$i$, the intersection matrix of~$D$ is negative-definite.
\item $l=2$ and~$k_0=1$. 
\end{itemize}

Either every irreducible component of~$\Gamma$ is a rational orbifold or every irreducible component of~$\Gamma$ is of type~$(\infty,\infty)$. From~(\ref{posram}), for $q\in C_i$, $q\notin\{p_i,p_{i-1}\}$, $\mathrm{ind}(C_i,q)>0$. We claim that~$q$ is an \emph{affine regular} point. The only other local model inducing a positive ramification index is, according to Table~\ref{table:combinatorics}, a \emph{finite ramification} point. However, such points have always a second separatrix, proving our claim. In particular, the only singular points of~$\mathcal{F}_X$ are the intersection points of two irreducible components of~$D$ and, for~$q\in C_i$, $q\notin\{p_i,p_{i-1}\}$, $\mathrm{ind}(C_i,q)=1$.\\

If~$C_i$ is a rational orbifold for every~$i$, let~$\mu_i>0$ denote its type. Since for every $q\notin\{p_i,p_{i-1}\}$, $\mathrm{ind}(C_i,q)=1$, we must have (up to changing the orientation of the cycle), $\mathrm{ind}(C_i,p_i)=\mu_i$ and $\mathrm{ind}(C_i,p_{i-1})=-\mu_{i}$. According to the Camacho-Sad formula and the local form of the \emph{finite ramification} points, for every~$i$,
\begin{equation}\label{sys1}\mu_{i-1}-k_i\mu_i+\mu_{i+1}=0.\end{equation}
If~$\Gamma$ is a cycle of rational curves of type~$(\infty,\infty)$ and~$X$ is an \emph{infinite ramification} at every point, if~$\mu_i=\mathrm{ord}(X,C_i)$ ($\mu_i>0$), the Camacho-Sad
relation at~$p_i$ reads~(\ref{sys1}). If~$\Gamma$ is a cycle of rational curves of type~$(\infty,\infty)$ and~$X$ is singular non-degenerate at each point, if~$\mu_i\in\mathbf{C}^*$ is the eigenvalue of the
restriction of~$X$ to~$C_i$ at~$p_i$, $-\mu_i$ is the eigenvalue of the
restriction of~$X$ to~$C_i$ at~$p_{i-1}$. The Camacho-Sad
formula at~$p_i$ gives again~(\ref{sys1}).

As a system of~$l$ linear equations in the~$l$ variables~$\mu_i$, the system~(\ref{sys1}) is given by the intersection matrix of~$D$ and cannot have any non-trivial solution unless the matrix is not negative definite, this is, unless~$k_i=2$ for every~$i$ or~$l=2$ and~$k_0=1$. In the first case, for~$Z=\sum_i C_i$, we have~$Z^2=0$. In the second case, the equations read~$\mu_0=2\mu_1$ and~$k_1\mu_1=2\mu_0$ and thus~$k_1=4$: for~$Z=2C_0+C_1$, $Z^2=0$. This finishes the proof of Proposition~\ref{propcycles} when~$\Gamma$ reduces to the cycle.

\paragraph{If~$\Gamma$ does not reduce to the cycle}  We have the following Lemma:

\begin{lemma}\label{notreeram}Let~$S$ be a (non-singular) surface, $X$ a reduced
semicomplete vector field in~$S$ and~$T\subset S$ be a divisor invariant
by~$\mathcal{F}_X$ whose  dual graph is a tree. Let~$\gamma_0\not\subset T$  be
a germ of curve invariant by~$\mathcal{F}_X$ intersecting~$T$ transversely at
some point~$p_0$ in the irreducible component~$C_1\subset T$ and such
that~$\mathrm{ind}(\gamma_0,p_0)\geq 1$. Suppose that~$T$ has no other
separatrix. Then, every irreducible component of~$T$ is a rational orbifold and,
\begin{itemize}
 \item if~$\mathrm{ind}(\gamma_0,p_0)=1$, $T$ may be collapsed to a non-singular point, where the vector field is still reduced;
 \item if~$\mathrm{ind}(\gamma_0,p_0)>1$ and~$(X,T)$ is minimal good, $T$ has no branching points.
\end{itemize}
\end{lemma}
\begin{proof} Let~$C_{n+1}\subset T$ represent a vertex of degree one (extremal vertex) in~$\Gamma'$, the dual graph of~$T$. Let~$C_1\stackrel{p_1}{\text{---}}C_2\stackrel{p_2}{\text{---}}\cdots \stackrel{p_{n}}{\text{---}}
C_{n+1}$ be the unique monotone path joining~$C_1$ to~$C_{n+1}$. Since $\mathrm{ind}(\gamma_0,p_0)>0$, the local model of~$X$ at~$p_0$ is a \emph{finite ramification} and~$\mathrm{ind}(C_1,p_0)<0$. This implies (Remark~\ref{ratorb}) that~$C_1$ is a rational orbifold, that~$\mathrm{ind}(C_1,p_1)>0$, and that the local model of~$X$ at~$p_1$ is a \emph{finite ramification}. By repeating this argument, we conclude that all the vertices in~$\Gamma'$ are rational orbifolds, that the local model of~$X$ at the intersection of two irreducible components of~$\Gamma'$ is a \emph{finite ramification} and that the local model of~$X$ at the other points is an \emph{affine regular} one.

We will begin by proving the Lemma in the particular case where~$\Gamma'$ has no branching points. We will thus suppose that~$T$  has the form~$C_1\stackrel{p_1}{\text{---}}C_2\stackrel{p_2}{\text{---}}\cdots \stackrel{p_{n}}{\text{---}}
C_{n+1}$, with~$C_i$ a rational orbifold of type~$m_i$  (we necessarily have~$m_{n+1}=1$).  Let~$m_0=\mathrm{ind}(\gamma_0,p_0)$.  Let~$k_i=-C_i^2$. From the
Camacho-Sad formula and the fact that the Camacho-Sad index of a \emph{finite ramification} point is strictly negative,  $k_i>0$. If~$k_i=1$ for some~$i$ (if~$C_i$ is an exceptional curve of the first kind), it may be blown down. After blowing down this curve, the length of~$T$ decreases and all the hypothesis are still satisfied. In this way, we may continue blowing down the exceptional curves of the first kind until~$T$ is collapsed to a point (necessarily an \emph{affine regular} one) or until~$k_i\geq 2$ for every~$i$. In order to prove the Lemma in this particular case, we must prove that, when~$k_i\geq 2$ for every~$i$, $m_0>1$. If~$n=0$, the self-intersection of~$C_{1}$ is~$-m_0/m_{1}=-m_0$ and thus~$m_0>1$. If~$n>0$,
the Camacho-Sad relations give~$m_ik_i=m_{i-1}+m_{i+1}$  for~$1\leq i\leq n$ and~$m_{n+1}k_{n+1}=m_{n}$ (this is, $k_{n+1}=m_{n}$). Adding these equations, we obtain
$$
\sum_{i=1}^{n}  m_ik_i=  m_{0}+m_{1}+2\sum_{i=2}^{n-1} m_i+m_{n}+m_{n+1}.
$$
From this and~$m_{n+1}=1$, $1+m_{0}-m_{1}-m_{n}=\sum_{i=1}^{n}(k_i-2)m_i\geq 0$.
Since~$m_{n}=k_{n+1}\geq 2$, $m_0>m_1$ and thus~$m_0>1$. This proves the Lemma when~$\Gamma'$ has no branching points.

For the general case, let~$C_j$ be a vertex of degree~$\delta_j>2$. Suppose, furthermore, that any monotone path from~$C_1$ to a vertex of degree one passing through~$C_j$ meets no vertices of degree strictly greater than~$2$ after~$C_j$. Beyond~$C_j$, $\Gamma'$ is given by connected trees~$\Gamma'_1,\ldots, \Gamma'_{\delta_j-1}$ within~$\Gamma'$, representing divisors intersecting~$C_j$ at the points~$q_1,\ldots, q_{\delta_j-1}$. By hypothesis, the trees~$\Gamma'_i$ have no branching points and, by the particular case of the Lemma (applied to the divisor represented by~$\Gamma'_i$ and the invariant curve given by the germ of~$C_j$ at~$q_i$), $\mathrm{ind}(C_j,q_i)>1$. But since~$C_j$ is a rational orbifold, it only has one point of ramification index greater than one and hence~$\delta_j=2$. This contradiction proves the Lemma. \end{proof}

Let us come back to the proof of Proposition~\ref{propcycles}. Since~$\Gamma_0$ is the unique cycle, $\Gamma$ is obtained from~$\Gamma_0$ by attaching some trees. Let~$C_0$ be an irreducible component belonging to the cycle, let~$p_0\in C_0$ and suppose that a tree~$T$ is attached to~$C_0$ at~$p_0$. By the previous Lemma (applied to~$T$ and the invariant curve~$C_0$ intersecting~$T$ at~$p_0$),  $\mathrm{ind}(C_0,p_0)>1$ and, since~$C_0$ cannot be of type~$(\infty,\infty)$, it is a rational orbifold. This implies that all irreducible components of the cycle are rational orbifolds and that at the intersection points of two irreducible components, we have a \emph{finite ramification point}. By the Lemma, the only vertices in~$\Gamma$ with degree greater than two have degree three and belong to the cycle. This finishes the proof of Proposition~\ref{propcycles}.\\

So far, we have not proved that the last possibility of Proposition~\ref{propcycles} may actually happen. As we will see later (Remark~\ref{negdef}), such vector fields and divisors are exactly the ones appearing in Kato surfaces. For the time being, let us exhibit two cases related to this last possibility where all the combinatorial data may be realized:

\begin{table}
a) 
\begin{tabular}{|c|c|c|r|}
\hline
$i$ & type $C_i$ & $\mathrm{ord}(X,C_i)$ & $C_i^2$ \\  \hline
$0$ & $3$ & $1$ &   $-3$  \\
$1$ & $2$ & $1$ &   $-2$ \\ 
$2$ & $1$ & $0$ &   $-3$ \\ \hline
\end{tabular}\;
b)
\begin{tabular}{|c|c|c|r|}
\hline
$i$ & type $C_i$ & $\mathrm{ord}(X,C_i)$ & $C_i^2$ \\  \hline
$0$ & $2$ & $1$ &   $-5$  \\
$1$ & $3$ & $2$ &   $-1$ \\ 
$2$ & $1$ & $0$ &   $-2$ \\ \hline
\end{tabular}
\caption{Some admissible combinatorics for~$D$}\label{tabexocyc}
\end{table}

\begin{exem}\label{exocyc} In the simplest case, the cycle is formed by two curves, $C_0$ and~$C_1$ and there is only one tree (consisting of one curve, $C_2$) attached to, say, $C_0$. For the values of the types, the orders and the self-intersections in Table~\ref{tabexocyc}, the reciprocity relation~(\ref{reciprocity}) holds at the three points of intersection. In both cases, $C_2$ is a rational orbifold of type~$1$ attached to~$C_0$ at a point of ramification index~$-1$ for~$C_2$ and greater than one for~$C_0$. Within~$C_0$, at the points of intersection with~$C_1$, one of the ramification indices is negative and the other one equals~$1$. In both cases, the intersection form is negative definite. In the second case, $C_1$ may be contracted to a point and the cycle reduces to a curve with a node (although, in this case, the vector field will no longer be minimal good in the sense of Definition~\ref{def:reduced}). \end{exem}

\begin{rema}\label{notreerem} In Lemma~\ref{notreeram}, $X$ need not be semicomplete, but only be so in a neighborhood of each invariant divisor. The same conclusion (with the same proof) holds.
 \end{rema}

\section{Vector fields without separatrices and Kato surfaces}\label{sec:kato}

Theorem~A will be proved in this Section. We begin by recalling the construction of Kato surfaces, following Kato \cite{kato} and Dloussky~\cite{dloussky-kato}. Let~$\widehat{S}$ be a non-singular surface and~$\widehat{D}$ a divisor in~$\widehat{S}$ that may be contracted to a non-singular point~$\widehat{\pi}:(\widehat{S},\widehat{D})\to(\mathbf{C}^2,0)$. Let~$q\in \widehat{D}$ and consider a germ of biholomorphism~$\widehat{\sigma}:(\mathbf{C}^2,0)\to (\widehat{S},q)$. To the \emph{Kato data} $(\widehat{\pi},\widehat{\sigma})$, we may associate a Kato (compact complex) surface in the following way. Let~$\epsilon>0$ be sufficiently small. Let
\begin{equation}\label{ball-sphere} B_\epsilon=\{(z,w);\;|z|^2+|w|^2<\epsilon\},\; \Sigma_\epsilon=\partial B_\epsilon.\end{equation}
The manifold-with-boundary~$M=\widehat{\pi}^{-1}(B_\epsilon \cup \Sigma_\epsilon)\setminus \widehat{\sigma} (B_\epsilon)$, has two boundary components, $\widehat{\pi}^{-1}(\Sigma_\epsilon)$ and~$\widehat{\sigma}(\Sigma_\epsilon)$. When identifying the first to the second by~$\widehat{\sigma}\circ\widehat{\pi}$,  we obtain a compact complex surface~$K$ (independent of~$\epsilon$). Upon contracting the exceptional curves within~$K$ we obtain a minimal compact complex surface~$K_0$. This is the \emph{Kato surface}  associated to~$(\widehat{\pi},\widehat{\sigma})$. If a vector field~$\widehat{X}$ is defined in~$\widehat{S}$ and if~$(\widehat{\sigma}\circ\widehat{\pi})_*\widehat{X}=\widehat{X}$, the surface~$K$ is naturally endowed with a holomorphic vector field~$Z$, which induces a holomorphic vector field~$Z_0$ in~$K_0$. In this case, $\widehat{f}=\widehat{\pi}\circ\widehat{\sigma}:(\mathbf{C}^2,0)\to(\mathbf{C}^2,0)$ preserves the vector field~$\widehat{\pi}_*\widehat{X}$ in~$(\mathbf{C}^2,0)$.

In the case where~$q$ belongs to the smooth part of~$\widehat{D}$ (which is the case that will concern us), each irreducible component of the support of~$\widehat{D}$ is, in a natural way, associated to an irreducible curve in~$K$. The components of~$\widehat{D}$ that do not contain~$q$ are naturally embedded in the surface~$K$. The component~$C$ of~$\widehat{D}$ that contains~$q$ gives rise to a curve in~$K$, obtained by gluing, via~$\widehat\sigma$, $C\setminus \widehat\sigma(B_\epsilon)$ with~$\widehat{\sigma}^{-1}(C)$. There are no further algebraic curves in~$K$.\\

In order to prove Theorem~A, we will start by establishing its germified version:

\begin{theorem}\label{main-local} Let~$(S_0,\varpi)$ be a germ of singular
surface and~$X_0$ be a semicomplete holomorphic vector field in~$(S_0,\varpi)$.
If the foliation induced by~$X_0$ has no separatrix through~$\varpi$ then there
exists a Kato surface~$K_0$, with maximal reduced divisor of rational
curves~$D$, a vector field~$Y_0$ in~$K_0$  and a minimal
resolution $\pi:(K_0,D)\to(S_0,\varpi)$ such that~$\pi_*(Y_0)=X_0$.
\end{theorem}

Let us begin the proof of this Theorem. Let~$\pi:(S,D,X)\to(S_0,\varpi,X_0)$ be a minimal good resolution in the sense of Definition~\ref{reducedvf}. Let~$\Gamma$ be the dual graph of~$D$. The vector field~$X$ is semicomplete in a neighborhood of~$D$, as discussed in~\S\ref{ssscias}. By Camacho's theorem~\cite{camacho-singular}, since~$\mathcal{F}_{X_0}$ has no separatrix through~$p$, $\Gamma$ has at least one cycle, $\Gamma_0$. Since~$D$ may be collapsed to a singular point, its intersection form is negative definite. Hence, by Proposition~\ref{propcycles},  every irreducible component of~$D$ is rational,
$\Gamma$ has a unique cycle and~$\Gamma$ contains, beyond the cycle, some trees that do not ramify and that are attached to the cycle (and thus~$D$ resembles the maximal divisor of rational curves in an intermediate Kato surface).
From the  combinatorial description of~$(S,X,D)$ done in the previous Section, we will construct a Kato data~$(\widehat{\pi},\widehat{\sigma})$ and an embedding of the minimal resolution of~$(S_0,\varpi)$ into the corresponding Kato surface. The reader is invited to have in mind the divisors and vector fields of Example~\ref{exocyc}. \\

Let~$C_0$  be an irreducible component of~$D$ corresponding to a vertex of degree three in the dual graph (it belongs to the cycle and there is a tree attached to it). The affine structure makes $C_0$ a rational orbifold. By Lemma~\ref{notreeram}, the point in~$C_0$ where the tree is attached has ramification index greater than one. There is thus a point~$p\in C_0$ such that~$\mathrm{ind}(C_0,p)=1$ and such that there exists an irreducible component~$C_1$ of~$D$ that belongs to the cycle and that intersects~$C_0$ at~$p$ (there is a third point in~$C_0$ where the affine structure has negative ramification index and where $C_0$ intersects an irreducible component of~$D$ belonging to the cycle: the case where this irreducible component is still~$C_1$ is not excluded). Let~$C_2$, \ldots, $C_n$ be the remaining irreducible components of~$D$. 

There are coordinates~$(x,y)$ around~$p$ where~$X$ is given by
$$ x^ny^{mn-1}\left(-mx\del{x}+y\del{y}\right), $$
where $n=\mathrm{ord}(C_0,p)$ and~$m=-\mathrm{ind}(C_1,p)$. In these coordinates, $C_0$ is defined by~$x=0$ and~$C_1$ by~$y=0$. Let~$\psi:(S,p)\to(\mathbf{C}^2,0)$ the holomorphic (non-invertible) mapping given by~$\psi(x,y)=(xy^m,y)$. If~$(z,w)$ are coordinates around~$(\mathbf{C}^2,0)$ and if~$Y=z^n\partial /\partial w$, $\psi_*X=Y$.

For each irreducible component~$C_i$ of~$D$ let~$U_i$ be a small (real) tubular neighborhood of~$C_i$. In the disjoint union~$\sqcup_i U_i$ identify the points that are equal in~$S$ \emph{except} those in the connected component of~$U_0\cap U_1$ containing~$p$. Let~$S^\sharp$ denote the resulting surface, $\pi^\sharp:S^\sharp\to S$ the natural immersion and~$X^\sharp$  the vector field in~$S^\sharp$ induced by~$X$. Let~$C_i^\sharp\subset S^\sharp$ be the compact curve coming from~$C_i$ in~$U_i$ and, for~$i\in\{0,1\}$, let~$p^\sharp_i\in S^\sharp$ be the preimage of~$p$ coming from~$U_i$. The support of the reduced divisor~$\sum C_i^\sharp$ has no  cycles. There are two separatrices for~$X^\sharp$ that are not contained in the divisor. One of them, $\gamma_1$, comes from~$C_1\cap U_0$ and passes through~$p^\sharp_0\in S^\sharp$. The other, $\gamma_0$, comes from~$C_0\cap U_1$ and passes through~$p^\sharp_1$.

Consider, in a neighborhood of~$p^\sharp_0$, the mapping~$\psi_0=\psi\circ\pi^\sharp$. It is a biholomorphism in the complement of the separatrix~$\gamma_1$. Let~$\widehat{S}$ be the surface obtained by from~$S^\sharp$ by removing~$\gamma_1$ and gluing back a neighborhood of the origin of~$\mathbf{C}^2$ via~$\psi_0$. We will denote by~$q\in \widehat{S}$ be the point corresponding to the origin in~$\mathbf{C}^2$ and by~$\widehat{\sigma}:(\mathbf{C}^2,0)\to(\widehat{S},q)$ the tautological mapping.  The surface~$\widehat{S}$ has a naturally defined vector field~$\widehat{X}$ obtained by the identification of~$X^\sharp$ in~$S^\sharp\setminus \gamma_1$ and~$Y$ on~$\mathbf{C}^2$   (after the surgery, the vector field around~$q$ is \emph{affine regular}). Let~$\kappa:S^\sharp\to \widehat{S}$ be the induced map (it collapses~$\gamma_1$ to a point and is an embedding in restriction to the complement of this curve).  Let~$\widehat{C_i}=\kappa(C_i^\sharp)$ and define~$\widehat{\gamma}_0$  and~$\widehat{p}_1$ similarly.
 Let~$\widehat{D}=\cup_{i=0}^n \widehat{C}_i$.

We affirm that $\widehat{D}$ may be contracted to a non-singular point in a surface. The induced affine structure makes every irreducible component of~$\widehat{D}$ a rational orbifold and the dual graph of~$\widehat{D}$ is a tree. The arc~$\widehat{\gamma}_0$ intersects~$\widehat{C}_1$ at~$\widehat{p}_1$ and $\mathrm{ind}(\widehat{\gamma}_0,\widehat{p}_1)=1$. We are in the setting of Lemma~\ref{notreeram}: we have a tree of rational orbifolds, $\widehat{D}$, rooted at an invariant arc, $\widehat{\gamma}_0$, whose affine structure has a positive ramification index and there are no further separatrices of~$\widehat{D}$. Since $\mathrm{ind}(\widehat{\gamma}_0,\widehat{p}_1)=1$, Lemma~\ref{notreeram} implies that~$\widehat{D}$ may be collapsed to a non-singular point in a surface (we do not know a priori if the vector field~$\widehat{X}$ in~$\widehat{S}$ is still semicomplete; however, we know that it is locally so and, by Remark~\ref{notreerem}, Lemma~\ref{notreeram} gives  still the desired result). In a 
neighborhood 
of~$\widehat{p}_1$, the contraction of~$\widehat{D}$ may be given by~$\psi\circ\pi^\sharp\circ\kappa^{-1}$, which maps~$\widehat{X}$ to~$Y$. This establishes coordinates around the blowdown of~$\widehat{D}$ or, equivalently, explicits a map~$\widehat{\pi}:(\widehat{S},\widehat{D})\to(\mathbf{C}^2,0)$ contracting~$\widehat{D}$ to a point. \\

The couple~$(\widehat{\pi},\widehat{\sigma})$ is a Kato data and produces, as explained at the beginning of this section, a compact surface~$K$ whose minimal model is a Kato surface~$K_0$. Tautologically, $(\widehat{\sigma}\circ\widehat{\pi})_*\widehat{X}=\widehat{X}$ and hence~$K$ (resp.~$K_0$) is naturally endowed with a vector field~$Z$ (resp.~$Z_0$). The complete construction is illustrated in Figure~\ref{fig:reconstructing}. \\

\begin{figure}
\begin{center}
\includegraphics[scale=0.7]{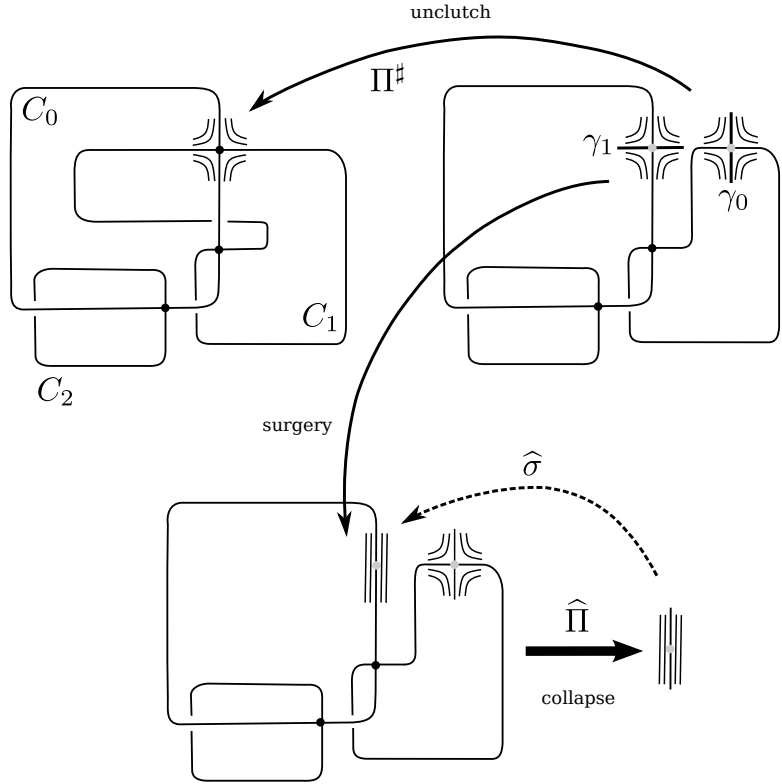}
\end{center}
\caption{Finding the Kato data}\label{fig:reconstructing}
\end{figure}

We will now embed~$S$ into~$K$ while mapping~$X$ to~$Z$. Consider, for some sufficiently small~$\epsilon$, the sphere $\Sigma_\epsilon$ as in formula~(\ref{ball-sphere}). In a neighborhood of~$p$, cut~$S$ along~$\psi^{-1}(\Sigma_\epsilon)$ in order to produce a (non-compact) manifold-with-boundary~$N$ with two boundary components, one corresponding to the interior  of~$\Sigma_\epsilon$, and one corresponding to its exterior. There is a natural identification between these two boundary components (two points in different boundary components are identified if they correspond to the same point in~$S$). There is a unique lift of~$N$ into~$S^\sharp$ that does not intersect~$\gamma_1$. It is still an embedding and, moreover, remains an embedding after composition with~$\kappa$. This gives an embedding~$j:N\to\widehat{S}$. Through this embedding, one of the boundary components of~$j(N)$ lies within~$\widehat{\pi}^{-1}(\Sigma_\epsilon)$; the other, within~$\widehat{\sigma}(\Sigma_\epsilon)$. By construction, the 
identification of $\widehat{\pi}^{-1}(\Sigma_\epsilon)$ and $\widehat{\sigma}(\Sigma_\epsilon)$ that produces the compact surface~$K$ identifies tautologically the boundary components of~$N$. This produces an embedding~$i:S\to K$ that maps~$X$ to the globally defined vector field~$Z$. If~$P:(S_\mu,D_\mu)\to(S_0,q)$ is the minimal resolution, there is a map~$\Theta:(S,D)\to (S_\mu,D_\mu)$ that may be factored as a sequence of blowdowns of exceptional curves.  Upon contracting the exceptional curves in~$K$ following the same pattern, we obtain an embedding of~$(S_\mu,D_\mu)$ into the Kato surface~$(K_0,D_0)$. This proves Theorem~\ref{main-local}.

\begin{figure}
\begin{center}
\includegraphics[scale=0.7]{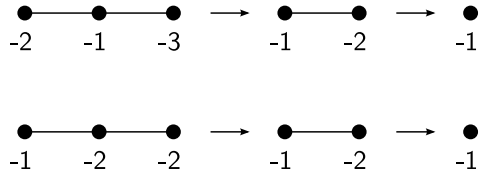}
\end{center}
\caption{Collapsing the divisors of Example~\ref{exocollapse}}\label{fig:contracting}
\end{figure}

\begin{table}
a) 
\begin{tabular}{|c|c|c|r|}
\hline
curve & type  & $\mathrm{ord}(X,\cdot)$ & self-int. \\  \hline
$\widehat{\gamma}_0$ & --- & $1$ &   ---  \\
$\widehat{C}_1$ & $2$ & $1$ &   $-2$ \\ 
$\widehat{C}_0$ & $3$ & $1$ &   $-1$  \\
$\widehat{C}_2$ & $1$ & $0$ &   $-3$ \\ \hline
\end{tabular}\\

b)
\begin{tabular}{|c|c|c|r|}
\hline
curve & type  & $\mathrm{ord}(X,\cdot)$ & self-int. \\  \hline
$\widehat{\gamma}_0$ & --- & $1$ &   ---  \\
$\widehat{C}_1$ & $3$ & $2$ &   $-1$ \\ 
$\widehat{C}_0$ & $2$ & $1$ &   $-2$  \\
$\widehat{C}_2$ & $1$ & $0$ &   $-2$ \\ \hline
\end{tabular}
\caption{Some instances of the construction}\label{tabsec}
\end{table}

\begin{exem}\label{exocollapse} For the divisors~$D$ of Example~\ref{exocyc}, the divisors~$\widehat{D}$, given by~$\widehat{C}_1\text{---}\widehat{C}_0\text{---}\widehat{C}_2$, are those whose weighted dual graphs appear in Figure~\ref{fig:contracting}. The figure describes the ways in which the successive contraction of the exceptional curves of the first kind leads to the contraction of all of~$\widehat{D}$. The separatrix~$\widehat{\gamma}_0$ intersects the component~$\widehat{C}_1$, corresponding in the figure to the leftmost vertices, and~$\mathrm{ind}(\widehat{\gamma}_0,\widehat{\gamma}_0\cap \widehat{C}_1)=1$. Table~\ref{tabsec} shows, in each case, the relevant data for the irreducible components of the divisor.
\end{exem}

To go from Theorem~\ref{main-local}, the germified version of Theorem~A, to the global one, we will use the dynamics of vector fields on Kato surfaces. According to~\cite[Thm.~2.14]{dloussky-vf1}, the maximal reduced divisor of rational curves~$D_0$ of~$K_0$ is an \emph{attractor} for the flow of~$Z_0$: every integral curve of~$Z_0$ that is not contained in~$D_0$ contains~$D_0$ in its closure. We may precise further the way in which the integral curves of~$Z_0$ accumulate to~$D_0$ (see also~\cite[Ex.~2.2]{brunella-kato}):
\begin{prop}\label{domainsvfkato} Let~$K_0$ be an intermediate Kato surface, $D_0$ its maximal reduced divisor of rational curves and $Z_0$ a holomorphic vector field in~$K_0$. Every neighborhood of~$D_0$ contains a neighborhood~$U$ of~$D_0$ such that if~$x\notin D_0$ and~$\phi:\mathbf{C}\to K$ is the solution of~$Z_0$ with initial condition~$x$ then~$\phi^{-1}(U)\subset\mathbf{C}$ is a connected set such that any connected component of its complement is compact.
\end{prop}
\begin{proof} Let~$\widehat{\pi}:(\widehat{S},\widehat{D})\to(\mathbf{C}^2,0)$ and~$\widehat{\sigma}$ be the corresponding Kato data and let~$\widehat{f}:(\mathbf{C}^2,0)\to (\mathbf{C}^2,0)$ be the germ~$\widehat{\pi}\circ \widehat{\sigma}$. There exists a vector field~$Y$ (induced by~$Z$) such that~$\widehat{f}_*Y=Y$. In suitable coordinates, we may suppose that~$Y=z^n\partial/\partial w$ for some~$n>0$. It has the first integral~$z$. Since $\widehat{f}$ preserves this vector field, it must be, up to a change of coordinates (preserving~$Y$), of the form \begin{equation}\label{presvf}\widehat{f}(z,w)=(z^{k+1},z^{nk}w+\tau(z)),\end{equation}
for some~$k\geq 1$ (Favre gave normal forms for such contracting germs and the above ones belong to the \emph{special} case of Class~4 in~\cite{favre-kato}; yet, the above formula will suffice for our needs). Let~$M^*=(B_\epsilon\cup \Sigma_\epsilon\setminus \widehat{f}(B_\epsilon))\cap\{z\neq 0\}$. It is a manifold-with-boundary with boundary components contained in~$\Sigma_\epsilon$ and~$\widehat{f}(\Sigma_\epsilon)$. Consider, within~$M^*$, the orbit~$O_\delta$ of~$Y$ given by~$z^{-1}(\delta)$. Suppose that~$\delta$ is small enough so that~$O_\delta$ intersects the two boundary components of~$M^*$. Close to~$\Sigma_\epsilon$,  the flow of~$Y$ restricted to~$M^*$ is defined in the interior of round disks in~$\mathbf{C}$. Close to~$\widehat{f}(\Sigma_\epsilon)$, the solutions of~$Y$ are defined in the exterior of round disks in~$\mathbf{C}$. The parametrization of~$O_\delta$ induced by~$Y$ is thus defined in a domain~$\Omega_\delta\subset\mathbf{C}$, which is a  round disk deprived of some disjoint round 
disks (more than one by holonomy considerations).

Since~$\widehat{f}$ is a local biholomorphism away from~$y=0$, it identifies the two boundary components of~$M^*$ and produces a non-compact manifold (without boundary)~$K^\circ$. This manifold has one end and embeds naturally into the Kato surface~$K_0$ (via the embedding~$\widehat{\pi}^{-1}:M^*\to\widehat{S}$). Its image is the complement in~$K_0$ of~$D_0$. Through this embedding, the end of~$K^\circ$ gets compactified by~$D_0$.

By~(\ref{presvf}), the identification of the boundary components of~$M^*$ will map the boundary of~$O_\delta$ to the boundary of~$O_{\delta^{k+1}}$ in such a way that, when gluing the parametrizations given by~$Z_0$,  the outer boundary component of~$\Omega_\delta$ is glued to one of the inner boundary components of~$\Omega_{\delta^{k+1}}$. Hence, within~$K_0$ and as we approach~$D_0$, the domain where the solution of~$Z_0$ is defined contains the union~$\cup_i \Omega_{\delta^{(k+1)^i}}$ (see Figure~\ref{fig:domkato}).

Thus, for every neighborhood~$W$ of~$D_0$ there exists some~$\delta>0$ such that the image of~$B=\pi^{-1}(\{z; |z|\leq \delta \}\cap M^*)$ in~$K_0$ is contained in~$W\setminus D_0$. The neighborhood~$U$ of~$D_0$ given by the interior of~$B\cup D_0$ has the required properties.
\end{proof}

\begin{figure}
\begin{center}
\includegraphics[scale=0.7]{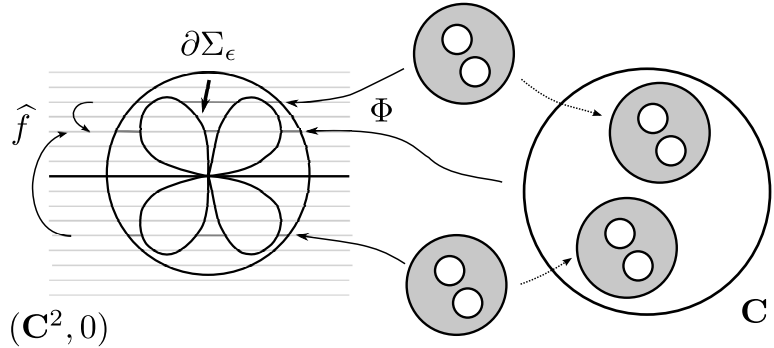}
\end{center}
 \caption{The domains in~$\mathbf{C}$ where the solutions of the vector field in~$M^*\subset\mathbf{C}^2$ are defined, and the identification between the boundaries of these domains induced by~$\widehat{f}$}\label{fig:domkato}
\end{figure}

\begin{proof}[Proof of Theorem~A] Let~$(S_0,\varpi,X_0)$ be a triple satisfying the hypothesis of Theorem~A. Let~$\pi:(S,D,X)\to(S_0,\varpi,X_0)$ be a minimal good resolution. By Theorem~\ref{main-local} there is a (blown-up) Kato surface~$K$ with a divisor~$D_K$ and a vector field~$Z$ and there is a neighborhood~$U$ of~$D_K$ such that there is a mapping~$\Psi_0:U\to S$ (the one guaranteed by Theorem~\ref{main-local}). We will suppose that~$U$ is a neighborhood like those produced by the proof of Proposition~\ref{domainsvfkato}.

Let~$\Phi_Z:\mathbf{C}\times K\to K$ and~$\Phi_X:\mathbf{C}\times S\to S$ be the corresponding flows. Since every integral curve of~$Z$ that is not contained in~$D_K$ contains~$D_K$ in its closure, for each~$q\in K$ there is some~$t_q\in\mathbf{C}$ such that~$\Phi_Z(t_q,q)\in U$. Define~$\Psi:K\to S$ as
\begin{equation}\label{katodiffeo}\Psi(q)=\Phi_X(-t_q,\Psi_0\circ \Phi_Z(t_q,q)).\end{equation}
We claim that~$\Psi$ is a biholomorphism, this is, that it is (i)~well-defined, (ii)~onto and~(iii) one-to-one. If~$V_q\subset\mathbf{C}$ is a connected neighborhood of~$t_q$ such that~$\Phi_Z(t_q',q)\in U$ for every~$t_q'\in V_q$ then~$t_q$ and~$t_q'$ define the same function in~(\ref{katodiffeo}). In particular, since~$\Phi_Z^{-1}(U)$ is connected (by Proposition~\ref{domainsvfkato}), the mapping is well defined. The image of~$\Psi$ is a compact set containing a neighborhood of~$D$ that is saturated by~$\mathcal{F}_X$ and is thus open. Hence, $\Psi$ is onto. It remains to prove that~$\Psi$ is one-to-one. Suppose that there are two different points in~$K$ having the same image under~$\Psi$ (they cannot belong both to~$U$ for~$\Psi$ is one-to-one in restriction to~$U$). If they belong to different orbits of~$Z$, the images of these orbits are the same one and thus there exist two orbits of~$X$ in restriction to~$\Psi_0(U)$ that get identified within~$S$. But this is impossible because of the nature of the 
domains where the flow of~$Z$ (in restriction to~$U$) is defined (Proposition~\ref{domainsvfkato}): two such domains must intersect. If the points are in the same orbit of~$Z$, one of the orbits of~$X$ has a period. However, by the description of the domains where the solution is defined (the solution is one-to-one in arbitrarily large domains), no period may arise. This finishes the proof of Theorem~A.\end{proof}

We may rephrase the passage from Theorem~\ref{main-local} to Theorem~A in the following way: \emph{If~$K$ is an intermediate Kato surface with vector field~$X$ and~$U$ is a neighborhood of the union of the rational curves in~$K$ then if~$Y$ is a complete vector field on the surface~$N$ and~$i:(U,X|_U)\to(N,Y)$ is an equivariant embedding, then~$i$ extends to a biholomorphism}. In general, given a semicomplete (and non-complete) holomorphic vector field~$X$ in some $n$-dimensional complex manifold~$M$, there exists  a \emph{completion of~$X$}, an $n$-dimensional manifold $N$, a complete holomorphic vector field~$Y$ on~$N$ and an equivariant embedding~$i:(M,X)\to (N,Y)$~\cite{palais} (there may be many of them and the manifold~$N$ may be non-Hausdorff). A remarkable fact is that, in the present situation, the dynamics of the vector field on a Kato surface force the uniqueness and the compactness of the completion.

\begin{rema}\label{negdef} In \S\ref{sec-cycles}, there remained the problem of understanding the vector fields related to the last possibility of Proposition~\ref{propcycles}: to understand the combinatorics that are realizable by a semicomplete vector field and, for example, to distinguish the cases where the corresponding intersection form is negative definite. Theorem~\ref{main-local} and its proof establish that such divisors are exactly the ones found in Kato surfaces admitting vector fields. By~\cite[Thm.~2.27]{dloussky-kato}, their intersection forms are negative definite.
\end{rema}

\section{Isolated equilibrium points in Stein surfaces}\label{sec:stein}

The \emph{Hirzebruch-Jung} or \emph{cyclic quotient} surface singularity~$A_{n,m}$ is the germ of analytic space obtained by taking the quotient of~$(\mathbf{C}^2,0)$ under the linear action of~$\mathbf{Z}/n\mathbf{Z}$
generated by
\begin{equation}\label{linearcation}(z,w)\mapsto (\xi z, \xi^m w),\end{equation}
for some primitive~$n^\text{th}$ root of unity~$\xi$ and some~$m<n$ such that~$(m,n)=1$. By writing
$$\frac{n}{m}=k_1-\cfrac{1}{k_2-\cfrac{1}{\ddots -\cfrac{1}{k_s}}},$$
we obtain a sequence of integers~$k_i\geq 2$. The exceptional divisor of a minimal resolution of~$A_{n,m}$ consists of~$s$ rational curves~$C_1,\ldots C_s$, such that~$C_i\cdot C_i=-k_i$, $C_i\cdot C_j=1$ if~$|i-j|=1$ and $C_i\cdot C_j=0$ otherwise. Reciprocally, if a singularity has a resolution of this form, it is analytically equivalent to~$A_{n,m}$, for the relatively prime integers~$n$ and~$m$ obtained from the sequence~$k_i$ via the above continued fraction~\cite[Ch.~III, \S 5]{BPV}.  \\

Let us proceed to the Proof of Theorem~B. Let~$S_0$, $\varpi$ and~$X_0$ be, respectively, a Stein surface, a singular point in~$S_0$ and a complete vector field on~$S_0$ like in the statement Theorem~B. Let~$\pi:(S,D,X)\to (S_0,\varpi,X_0)$ be a minimal good resolution. Every curve invariant by~$X$ intersecting~$D$ but not contained in it (a curve coming from a separatrix of~$X_0$ at~$\varpi$) does so transversely at a smooth point of~$D$.

Let~$\gamma:(\mathbf{C},0)\to(S,q)$ be a curve such that~$\pi\circ\gamma$ is a separatrix of~$X_0$ at~$\varpi$ (such a curve exists by Theorem~A). The restriction of~$X$ to the image of~$\gamma$ is a vector field of the form~$f(z)\partial/\partial z$ (with~$f$ not identically zero by hypothesis). Since this vector field is semicomplete, up to a reparametrization of~$\gamma$, the vector field is either of the form~$z^2\partial/\partial z$ or~$\lambda z\partial/\partial z$ for some~$\lambda\in\mathbf{C}^*$~\cite[\S 3]{rebelo}. In the first case, the separatrix may be parametrized by~$t\mapsto -t^{-1}$ and we must conclude that the orbit of~$X$ containing~$\gamma\setminus\{q\}$ has trivial stabilizer and is compactified by~$q$ into a rational curve within~$S$ which is not contained in~$D$, which is impossible (this is the only point where we use the Steinness assumption on~$S_0$;  more generally, we may suppose that~$S_0$ does not contain rational curves). The restriction of~$X$ to every separatrix is hence 
locally given by~$\lambda z\partial/\partial z$.

Let~$E$ be an irreducible component of~$D$ that is  not  invariant by~$\mathcal{F}$. By Table~\ref{table:combinatorics} and the fact that~$X$ has no zeros near~$D$ that are not contained in~$D$, at every point of~$E$, $X$ is, in suitable coordinates, $f(x,y)x\partial/\partial x$ for some non-vanishing function~$f$. There is thus a function~$\lambda:E\to\mathbf{C}^*$ that gives, for each~$p\in E$, the eigenvalue of the restriction of~$X$ to the orbit passing through~$p$. In the above local model, $f(x,y)x\partial/\partial x$, $E$ is given by~$x=0$ and~$g(y)=f(0,y)$. Hence, $\lambda$ is constant and the flow of~$X$, near~$E$, has period~$2i\pi\lambda^{-1}$. This proves Theorem~B in the case where not all irreducible components of~$D$ are invariant by~$\mathcal{F}$.

We will henceforth assume that~$D$ is invariant by~$\mathcal{F}$. Let~$\gamma$ be a separatrix intersecting an irreducible component~$C_1$ of~$D$ transversely at some point~$p_0$. We may now go through the list of local models in Table~\ref{table:combinatorics} and conclude that the only local models such that the restriction of~$X$ to~$\gamma$ is of the form~$\lambda z\partial/\partial z$, such that all the zeros of~$X$ (if any) are contained in~$D$ and such that~$D$ is invariant by~$\mathcal{F}$ are, up to an invertible multiplicative factor,
\begin{enumerate}
\item  $x(\lambda+\cdots) \partial/\partial x+y(\mu+\cdots)\partial/\partial y$, $\lambda,\mu\in\mathbf{C}$;
\item  $x(1+\nu y)\partial/\partial x+y^2 \partial/\partial y$, $\nu\in\mathbf{Z}$;
\end{enumerate}
with the separatrix being, in both cases, the curve~$\{y=0\}$ and~$D$ given by~$\{x=0\}$.

In the first case, $\mathrm{ind}(C_1,p_0)=-1$. We conclude that~$C_1$ is a rational orbifold of type~$1$, that it has no further singularities of~$\mathcal{F}$ and that~$D$ reduces to~$C_1$. A semicomplete vector field having this combinatorics in~$A_{n,1}$ is obtained by resolving the quotient of~$z\partial/\partial z+(w+z)\partial/\partial w$ under~(\ref{linearcation}), for~$n=-C_1^2$.

In the second case, $\mathrm{ind}(C_1,p_0)=\infty$.  Since the affine structure in~$C_1$ is induced by a non-zero vector field, $C_1$ is of type~$(\infty,\infty)$ and there is thus a point~$p_1\in C_1$ such that~$\mathrm{ind}(C_1,p_1)=\infty$.  If~$p_1$ is a \emph{singular non-degenerate} point and if~$D$ does not reduce to~$C_1$, there is a component~$C_2$ of~$D$ which intersects~$C_1$ at~$p_1$ and which is of type~$(\infty,\infty)$. Continuing this argument, we have,  within~$D$,   a maximal chain of the form
\begin{equation}\label{chain}C_1\stackrel{p_1}{\text{---}}C_2\stackrel{p_2}{\text{---}}\cdots \stackrel{p_{s-1}}{\text{---}}
C_{s},\end{equation}
with a point~$p_s\in C_s$, different from~$p_{s-1}$, where~$X$ vanishes, such that~$X$ is \emph{singular non-degenerate} at~$p_i$ for~$i=1,\ldots, p_{s-1}$, such that~$C_i$ is of type~$(\infty,\infty)$ for every~$i$ and such that either~$D$ reduces to the above chain or~$p_s$ is not a \emph{singular non-degenerate} point. Notice that, for every~$i$, the points of~$C_i$ other than~$p_{i-1}$ and~$p_i$ have positive ramification index. By Lemma~\ref{notreeram}, no separatrix or irreducible component of~$D$ may meet~$C_i$ at these points. If~$D$ reduces to (\ref{chain}),  $X$ is singular non-degenerate at~$p_s$, where it has another separatrix, and~$D$ is a Hirzebruch-Jung string.  A semicomplete vector field of this kind in~$A_{n,m}$ is obtained by resolving the quotient of~$\lambda z\partial/\partial z+w\partial/\partial w$  under the action~(\ref{linearcation}). If~$D$ does not reduce to (\ref{chain}), there is an irreducible component~$C_{s+1}$  of~$D$ that intersects~$C_s$ at~$p_s$. The vector field~$X$ 
must be a saddle-node at~$p_s$, for~$X$ is holomorphic and non-identically zero along~$C_s$. The component~$C_{s+1}$ must be a rational orbifold of order~$1$ and have no further singularities. The divisor~$D$ reduces to the Hirzebruch-Jung string $C_1\stackrel{p_1}{\text{---}}C_2\stackrel{p_2}{\text{---}}\cdots \stackrel{p_{s}}{\text{---}}
C_{s+1}$.  A semicomplete vector field of this kind in~$A_{n,m}$ is obtained by resolving the quotient of~$z\partial/\partial z+(mw+z^m)\partial/\partial w$ under~(\ref{linearcation}).  This finishes the proof of Theorem~B.\\

In the case where the flow of the vector field factors through and action of~$\mathbf{C}^*$,
Camacho, Movasati and Sc\'ardua~\cite{ca-mo-sca} proved that a holomorphic action of~$\mathbf{C}^*$ on a Stein surface with a dicritical singularity is holomorphically and equivariantly equivalent to an algebraic action of~$\mathbf{C}^*$ on an affine surface (a case widely studied by Orlik and Wagreich~\cite{ow}). A discussion around normal forms for the germs vector fields in~$A_{n,m}$ coming from non-degenerate ones in~$(\mathbf{C}^2,0)$ may be found in the work of S\'anchez-Bringas~\cite{sanchezbringas}.\\

Theorem~B generalizes Lemma~6.1 in~\cite{rebelo-realisation}, which affirms that a complete vector field on a non-singular Stein surface has two non-vanishing eigenvalues at an isolated equilibrium point.

\section{Vector fields on compact complex surfaces}\label{sec:comp}

We will now proceed to the proof of Theorem~C, by revisiting the strategy developed in~\cite{dloussky-vf1}, which resorts to the local theory of semicomplete holomorphic vector fields and uses their local models  up to biholomorphisms, as developed in~\cite{ghys-rebelo}, \cite{rebelo-mex} and \cite{rebelo-realisation}. On the one hand, our approach benefits from the proof of Theorem~A, which allows one to readily recognize Kato surfaces. On the other, by systematically adopting the bimeromorphic point of view, the list of local models of semicomplete vector fields becomes smaller (we need only consider the reduced ones). Finally, the use of the leafwise affine structure and its numerical invariants will allow us to deal more effectively with the combinatorics.\\

Let~$X$ be a holomorphic vector field on the compact complex (not necessarily minimal) surface~$S$. If there are infinitely many compact curves tangent to~$\mathcal{F}$, the latter has a first integral by a result of  Jouanoulou-Ghys~\cite{ghys-j}  (the generic level curve supports a non-identically zero vector field and is either rational or elliptic). We will henceforth suppose that there are only finitely many algebraic curves invariant by~$\mathcal{F}$.

Let~$D$ be the reduced divisor supported in the union of the algebraic curves invariant by~$\mathcal{F}$. Blow up as many singular points of~$X$ as necessary so that~$X$ becomes \emph{reduced} and~$(X,D)$ becomes \emph{minimal good} in the sense of Definition~\ref{def:reduced}.  

All the irreducible components of~$D$ are non-singular curves which are either rational or elliptic (for these are the only curves admitting uniformizable affine structures with singularities). Furthermore, we may suppose that all the rational curves have strictly negative self-intersection for, otherwise, the surface would be either rational or ruled~\cite[Prop.~4.3]{BPV}.

Arguing like in~\cite[Lemme~2.2]{dloussky-vf1}, if~$p\in S$ is a point where~$X(p)=0$ and if~$\gamma$ is a germ of curve invariant by~$\mathcal{F}$ such that~$X|\gamma$ is not identically zero, either the restriction of~$X$ to~$\gamma$ is equivalent to~$z^2\partial/\partial z$ or~$\lambda z\partial/\partial z$. We affirm that, in both cases, $\gamma$ is contained in an algebraic curve. In the first case, $\gamma$ is, as discussed in the previous Section, contained in an algebraic (rational) curve. In the second case, $\gamma$ is pointwise fixed by the flow of~$X$ in time~$2i\pi\lambda^{-1}$. The set of points of~$S$ fixed by~$X$ in time~$2i\pi\lambda^{-1}$ is a closed analytic subset of~$S$ which is not all of~$S$ since~$X$ induces an effective action of~$\mathbf{C}$. The curve~$\gamma$ is thus contained in an algebraic curve. 

In particular, if~$C$ is an irreducible component of~$D$ that is not invariant by~$\mathcal{F}$, every curve intersecting it must be algebraic (a case that has already been ruled out). Hence, we conclude,
\begin{itemize}
\item that every one-dimensional component of the locus of zeros of~$X$ is invariant by~$\mathcal{F}$, and
\item that  every germ of curve invariant by~$\mathcal{F}$ that intersects~$D$ is contained in~$D$.
\end{itemize}
Let~$D_0$ be a connected component of~$D$. If~$D_0$ contains a cycle then, by Proposition~\ref{propcycles}, either $D_0$ supports an effective divisor~$Z$ such that~$Z^2=0$ or, by Theorem~A and Remark~\ref{negdef}, $S$ is a Kato surface. We will henceforth suppose that~$D_0$ has no cycles.

\paragraph{With saddle-nodes or singular non-degenerate points} If the vector field is a saddle-node at the point~$p_0$ where the irreducible components~$C_0$, $C_1$ of~$D_0$ intersect, $\mathrm{ind}(C_1,p_0)=\infty$ and~$\mathrm{ind}(C_0,p_0)=-1$. The irreducible component~$C_0$ is necessarily a rational orbifold of type~$1$ carrying no further singular points of~$X$. The component~$C_1$ is of type~$(\infty,\infty)$, for its affine structure is induced by a non-vanishing vector field. There is thus a point~$p_1\in C_1$ where~$X$ vanishes and that is either a \emph{singular non-degenerate} or a \emph{saddle-node} point. If~$p_1$ is a \emph{singular non-degenerate} point, there is another component~$C_2$, intersecting~$C_1$ at~$p_1$, of type~$(\infty,\infty)$, along which~$X$ does not vanish, having another singular point~$p_2$ of~$X$. If~$p_1$ is a \emph{saddle-node}, there is another component~$C_2$ intersecting~$C_1$ at~$p_1$. Since~$\mathrm{ind}(C_2,p_1)=-1$, $C_2$ is a rational orbifold of type~$1$ 
having no further singularities of~$X$. We conclude that~$D_0$ is of the form $C_0\stackrel{p_0}{\text{---}}C_1\stackrel{p_1}{\text{---}}\cdots
\stackrel{p_l}{\text{---}} C_{l+1}$ ($l\geq 1$), with~$C_0$ and~$C_{l+1}$ rational orbifolds of type~$1$ and~$C_1$, \ldots, $C_l$ of type~$(\infty,\infty)$. The vector field~$X$ has a saddle-node at~$p_0$ and~$p_l$ and \emph{singular non-degenerate} points at~$p_1$, \ldots, $p_{l-1}$. Let~$\lambda_i\in\mathbf{C}^*$ be the eigenvalue of~$X$ at~$p_i$ within~$C_i$ (and thus~$-\lambda_i$ is the eigenvalue of~$X$ at~$p_{i-1}$ within~$C_i$). Let~$k_i=-C_i^2$ and suppose that~$k_i>0$. By the Camacho-Sad formula, the contribution of~$p_0$ (resp.~$p_l$) to the self-intersection of~$C_1$ (resp.~$C_l$) is zero (hence, if~$l=1$, $C_1^2=0$ so we will suppose that~$l>1$). By the Camacho-Sad formula, $k_1\lambda_1=\lambda_2$,
\begin{equation}\label{sn-cs}\lambda_i k_i=\lambda_{i-1}+\lambda_{i+1} \text{ for }i=2,\ldots, l-1, \end{equation} and~$k_l\lambda_l=\lambda_{l-1}$. Adding these equations, we get
\begin{equation}\label{impo}\sum_{i=1}^l(k_i-2)\lambda_i+\lambda_1+\lambda_l=0.\end{equation}
Up to dividing~$X$ by~$\lambda_1$, we may suppose that~$\lambda_1=1$ and thus that~$\lambda_2=k_1\in\mathbf{Z}$. From equation~(\ref{sn-cs})  for~$i=2$, we may solve  for~$\lambda_3$ and hence~$\lambda_3\in\mathbf{Z}$. Continuing this argument, we conclude that~$\lambda_i\in\mathbf{Z}$. All of them must be positive since~$-\lambda_i/\lambda_{i+1}\notin \mathbf{Q}^+$. From equation~(\ref{impo}), $k_i=1$ for some~$i$. This contradicts the minimal good character of~$(X,D)$. We conclude that such components appear only in the case~$l=1$, where we find a rational curve of vanishing self-intersection.

If there is a singular non-degenerate point~$p_0$, there are two invariant rational curves~$C_0$, $C_1$ through~$p_0$, whose ramification index at~$p_0$ is~$\infty$ and such that the vector field is not identically zero along them. Thus, there is another point~$p_1\in C_2$ where the affine structure has ramification index~$\infty$. Since~$D_0$ is free of cycles, this implies that there is a chain of curves of type~$(\infty,\infty)$ which must eventually have a \emph{saddle-node}, reducing this case to the previous one.

\paragraph{The other cases} If~$D_0$ contains an elliptic curve then the induced affine structure is non-singular and hence, by Lemma~\ref{notreeram}, $D_0$ reduces to an elliptic curve~$Z$ and there are no singularities of~$\mathcal{F}_X$ along the curve. By the Camacho-Sad formula, $Z^2=0$. In the other cases all the irreducible components of~$D_0$ are rational curves and $D_0$ is a tree. Label the irreducible components~$C_1$, \ldots, $C_n$ of~$D_0$ in such a way that, in the dual graph of~$D_0$, the subgraph generated by~$C_1, \ldots, C_k$ is connected for every~$k$, this is, for every~$i>1$ there is a unique~$\ell(i)<i$ such that~$C_i\cdot C_{\ell(i)}=1$.  The contribution of~$C_i\cap C_{\ell(i)}$ to the self-intersection of~$C_{\ell(i)}$ (by means of the Camacho-Sad formula) is a strictly negative rational~$r_i=\mathrm{CS}(\mathcal{F}_X,C_{\ell(i)},C_i\cap C_{\ell(i)})$, for we do not have neither \emph{saddle-nodes} nor \emph{singular non-degenerate points}, and all the singular points of~$\mathcal{F}$ are either \emph{finite} or~\emph{infinite ramification} points. Thus,  $C_j^2=r_j^{-1}+\sum_{k\in\ell^{-1}(j)}r_k$.

Let us show, following the proof of~\cite[Prop.~2.10]{dloussky-vf1}, that this implies that there is a divisor of vanishing self-intersection supported in~$D_0$. Let~$a_1=1$. Inductively, define $a_{j+1}=-r_{\ell(j+1)}a_{\ell(j+1)}$ (notice that~$a_{j+1}>0$) and let~$Z=\sum_{i}a_iC_i$. Then~$Z^2=0$ since
\begin{eqnarray*}Z\cdot C_j  =\sum_{i}a_iC_i\cdot C_j & = & a_jC_j^2+a_{\ell(j)}C_{\ell(j)}\cdot C_j+\sum_{k\in\ell^{-1}(j)}a_k C_k\cdot C_j \\
& = & a_j\left(C_j^2+\frac{a_{\ell(j)}}{a_j}+\sum_{k\in\ell^{-1}(j)}\frac{a_k}{a_j}\right)\\
& = & a_j\left(C_j^2-\frac{1}{r_j}-\sum_{k\in\ell^{-1}(j)}r_k\right)=0.
\end{eqnarray*}
Upon multiplying the~$\mathbf{Q}$\nobreakdash-divisor~$Z$ by some positive integer, we obtain the desired divisor. This proves Theorem~C.\\

We could precise further the nature of the divisor of vanishing self-intersection~$Z$ appearing in the statement of Theorem~C: it is either rational or a \emph{divisor of elliptic fiber type}, this is, it has the combinatorics of the (minimal good versions of the) divisors appearing in Kodaira's list of singular fibers in elliptic fibrations, like in~\cite[Thm.~A]{guillot-rebelo}. We will not pursue this direction.  \\

Our results may be used to give an alternative proof of Theorem~0.1 in~\cite{dloussky-vf2}, which completes the classification of minimal compact complex surfaces admitting vector fields.


\providecommand{\bysame}{\leavevmode ---\ }
\providecommand{\og}{``}
\providecommand{\fg}{''}
\providecommand{\smfandname}{\&}
\providecommand{\smfedsname}{\'eds.}
\providecommand{\smfedname}{\'ed.}
\providecommand{\smfmastersthesisname}{M\'emoire}
\providecommand{\smfphdthesisname}{Th\`ese}

\end{document}